\documentclass[10pt,a4paper]{amsart}
\usepackage[english]{babel}
\usepackage{amsmath,amsthm,mathrsfs,hyperref}
\usepackage{amssymb}
\usepackage{aliascnt}
\usepackage{xspace}
\DeclareMathOperator{\image}{''}

\DeclareMathOperator{\dom}{dom}
\DeclareMathOperator{\range}{range}
\DeclareMathOperator{\supp}{supp}
\DeclareMathOperator{\crit}{crit}

\DeclareMathOperator{\cf}{cf}
\DeclareMathOperator{\Add}{Add}

\DeclareMathOperator{\Hull}{Hull}
\DeclareMathOperator{\LSCP}{LSCP}
\newcommand{\GCH}{\mathrm{GCH}}
\newcommand{\ZFC}{\mathrm{ZFC}}
\newtheorem{theorem}{Theorem}
\newaliascnt{example}{theorem}

\aliascntresetthe{example}
\newaliascnt{fact}{theorem}

\aliascntresetthe{fact}
\newaliascnt{corollary}{theorem}
\newtheorem{corollary}[corollary]{Corollary}
\aliascntresetthe{corollary}
\newaliascnt{remark}{theorem}
\newtheorem{remark}[remark]{Remark}
\aliascntresetthe{remark}
\newaliascnt{lemma}{theorem}
\newtheorem{lemma}[lemma]{Lemma}
\aliascntresetthe{lemma}
\newaliascnt{claim}{theorem}
\newtheorem{claim}[lemma]{Claim}
\aliascntresetthe{claim}

\theoremstyle{definition}
\newaliascnt{definition}{theorem}
\newtheorem{definition}[definition]{Definition}
\aliascntresetthe{definition}
\newtheorem{question}{Question}
\newaliascnt{notation}{theorem}
\newtheorem{notation}[notation]{Notation}
\aliascntresetthe{notation}

\begin{document}
\title{Subcompact cardinals, type omission and ladder systems}
\author{Yair Hayut}
\address[Yair Hayut]{
Einstein Institute of Mathematics,
Edmond J. Safra Campus,
The Hebrew University of Jerusalem,
Givat Ram. Jerusalem, 9190401, Israel
}
\email[Yair Hayut]{yair.hayut@mail.huji.ac.at}
\thanks{The first author research was partially supported by the FWF Lise Meitner grant, 2650-N35}

\author{Menachem Magidor}
\address[Menachem Magidor]{
Einstein Institute of Mathematics,
Edmond J. Safra Campus,
The Hebrew University of Jerusalem,
Givat Ram. Jerusalem, 9190401, Israel
}
\email[Menachem Magidor]{mensara@savion.huji.ac.il}
\begin{abstract}
We provide a model theoretical and tree property-like characterization of $\lambda$-$\Pi^1_1$-subcompactness and supercompactness. We explore the behavior of these combinatorial principles at accessible cardinals.
\end{abstract}
\maketitle
\section{Introduction}
The study of very large cardinals and their connections to reflection
principles in infinitary combinatorics is a fruitful area of research that began
with the work of Erd\H{o}s, Tarski, Keisler, Scott and others \cite{KanamoriMagidor78}. Since Scott's work on measurable cardinals \cite{Scott}, large cardinal axioms are usually defined in terms of the existence of certain elementary embeddings between transitive models (see \cite{Gitman,Kanamori}). 
The study of elementary embeddings brings to light relationships between various large cardinals which are usually more difficult to derive using purely combinatorial arguments.
In this paper we will focus on the $\lambda$-$\Pi^1_1$-subcompact cardinals, which were isolated by Neeman and Steel in \cite{NeemanSteelSubcompact}. These cardinals can be viewed intuitively as a generalization of weak compactness to successor cardinals, or more precisely to $P_\kappa \kappa^+$. See \cite{Cody2020} for the analogous definition of the weakly compact filter and \cite{SchimmerlingZemanSquares} for thredability. In \cite{NeemanSteelSubcompact}, the terminology $\Pi^2_1$-subcompact is used to refer to what we denote by $\kappa^+$-$\Pi^1_1$-subcompact. 

We will provide two characterizations of $\lambda$-$\Pi^1_1$-subcompactness. The first one, which is discussed in Section \ref{section: type omission}, is model theoretical in nature and uses a mixture of compactness and type omission. This characterization is an strengthening of Benda's theorem from \cite{Benda78}. We modify Benda's original argument in order to remove the need of using infinitary logic and getting local equivalence. In \cite{Boney20}, Boney obtained a characterization of measurable, supercompact and huge cardinals using similar ideas --- a combination of compactness and type omission. 
The second characterization, discussed in Section \ref{section: ladder systems} is purely combinatorial, and can be viewed as a strengthening of a local instance of the strong tree property, together with inaccessibility, thus continuing the results of \cite{Jech1972, Magidor1974, Weiss} and others. 

The paper is organized as follows. In Section \ref{section: partial strong compactness} we review some facts about strong compactness and $\lambda$-$\Pi^1_1$-subcomapctness. In Sections \ref{section: type omission} and \ref{section: ladder systems} we provide characterizations of $\lambda$-$\Pi^1_1$-subcompactness. In Section \ref{section: aleph2}, we investigate the analog of the combinatorial principles that were defined in Section \ref{section: ladder systems} for $\aleph_2$, and show that the equivalence that holds at inaccessible cardinals consistently fails at $\aleph_2$.
\section{Strong compactness and \texorpdfstring{$\Pi^1_1$}{pi11}-subcompactness}\label{section: partial strong compactness}
In this section we will address some basic results regarding strongly compact and supercompact cardinals, as well as their local versions. The results in this section are not due to us, and are scattered through the literature. 

Keisler and Tarski \cite{KeislerTarski1963} defined a cardinal $\kappa$ to be strongly compact cardinal if every $\kappa$-complete filter can be extended to a $\kappa$-complete ultrafilter or equivalently if every $\mathcal{L}_{\kappa,\kappa}$-theory $T$ has a model provided that every subset $T' \subseteq T$ of size $<\kappa$ has a model. 

\begin{definition}
We say that a theory $T$ is \emph{$<\kappa$-satisfiable} if every subset of $T$ of size $<\kappa$ has a model. We say that a theory is \emph{satisfiable} if it has a model.
\end{definition}
Restricting the size of $T$ we obtain (consistently) a non-trivial hierarchy. 
\begin{definition}
Let $\kappa \leq \lambda$ be cardinals. We say that \emph{$\mathcal{L}_{\kappa,\kappa}$-compactness for languages of size $\lambda$} holds if every theory $T$ over a language of size $\lambda$ which is $<\kappa$-satisfiable is satisfiable.
\end{definition}
Localizing the equivalence for strong compactness, we obtain:
\begin{lemma}\label{lemma: localized strong compactness}
The following are equivalent for uncountable cardinals $\kappa \leq \lambda = \lambda^{<\kappa}$:
\begin{itemize}
\item $\mathcal{L}_{\kappa,\kappa}$-compactness for languages of size $\lambda$ holds.
\item For every transitive model $M$ of size $\lambda$ which is closed under $<\kappa$-sequences there are a transitive model $N$ with ${}^{<\kappa}N\subseteq N$,
an elementary embedding $j\colon M \to N$ with critical point $\kappa$, and 
an element $s\in N$ such that $j\image M \subseteq s$ and $|s|^N < j(\kappa)$.
\end{itemize}
\end{lemma}
See \cite{Hayut2019} or \cite{BuhagiarDzamonja} for a proof of this lemma. In the second clause, since $|s|^N \geq |s|^V \geq \lambda$, we get that $j(\kappa) > \lambda$. 

The elementary embeddings which are obtained from strong compactness have surprisingly weak implications in terms of reflecting properties. For example, the least measurable cardinal can be strongly compact. 

In order to obtain a stronger reflection we have to assume some form of \emph{normality}. 

\begin{definition}[Reinhardt, Solovay]
A cardinal $\kappa$ is $\lambda$-supercomapct if there is a fine and normal measure on $P_\kappa \lambda$. 
\end{definition} 

Equivalently, $\kappa$ is $\lambda$-supercompact if there is an elementary embedding $j\colon V\to M$, with critical point $\kappa$, such that $M$ is transitive, $\lambda < j(\kappa)$ and $j\image \lambda \in M$.

In contrast to the situation with strong compactness, if $\kappa$ is $2^{\kappa}$-supercompact, then there are many measurable cardinals below it. Let us take a closer look at this case. 
\begin{definition}[Jensen]
A cardinal $\kappa$ is $\lambda$-subcompact if for every $A \subseteq H(\lambda)$, there are $\bar\kappa, \bar\lambda$ and $\bar{A} \subseteq H(\bar{\lambda})$, and an elementary embedding:
\[j \colon \langle H(\bar{\lambda}),\in,\bar{A}\rangle \to \langle H(\lambda), \in A\rangle,\]
with $\crit j = \bar\kappa$, $j(\bar\kappa) = \kappa$. 
\end{definition}

A cardinal $\kappa$ is \emph{subcompact} if it is $\kappa^+$-subcompact. In the context of $\GCH$, one can easily verify that if $\kappa$ is $\kappa^+$-supercompact, then there are many subcompact cardinals below it. In the inner model context, a cardinal $\kappa$ is subcompact in an extender model $L[E]$, if and only if there are stationarily many $\alpha < \kappa^+$, such that $E_\alpha \neq \emptyset$ (using the Jensen-Friedman indexing) see \cite{SchimmerlingZemanSquares}. Jensen's definition follows a result of the second author \cite{Magidor1974}, in which he proved that a cardinal $\kappa$ is supercompact, if and only if it is $\lambda$-subcompact for all $\lambda \geq \kappa$.

The least subcompact cardinal is not measurable. Indeed, similarly to Woodin cardinals, subcompactness of $\kappa$ does not provide an elementary emebdding with critical point $\kappa$, but rather just many elementary embeddings that reach up to $\kappa$. 

By requiring the subcompact embeddings to resemble more the embeddings which are obtained from a supercompactness hypothesis we obtain a non-trivial hierarchy of strengthening of subcompactness.  Informally, we expect to get a parallel hierarchy of large cardinal axioms which accumulates to strong compactness. We expect that during the steps of those hierarchies, one would be able to move from the strong compactness side to the supercompactness side by adding a restriction. Borrowing from the characterizations using measures on $P_\kappa \lambda$, we call those missing ingredient \emph{normality assumptions}. 

So, we would like to isolate a normality property for the elementary embeddings which are obtained from partial strong compactness, as in Lemma \ref{lemma: localized strong compactness}. We will argue ahead that a natural candidate for the normalized partial strong compactness is:
\begin{definition}\label{definition:pi11-subcompact}
Let $\kappa \leq \lambda$ be cardinals.
$\kappa$ is \emph{$\lambda$-$\Pi^1_1$-subcompact} if for every $A \subseteq H(\lambda)$ and every $\Pi^1_1$-statement $\Phi$ such that $\langle H(\lambda), \in, A\rangle \models \Phi$, there are:
\begin{enumerate}
\item a pair of cardinals $\bar\kappa \leq \bar \lambda < \kappa$, 
\item a subset $\bar{A} \subseteq H(\bar{\lambda})$ such that $\langle H(\bar\lambda), \in \bar{A}\rangle \models \Phi$ and 
\item an elementary embedding:
\[j \colon \langle H(\bar\lambda), \in \bar{A}\rangle \to \langle H(\lambda), \in, A\rangle,\]
with critical point $\bar{\kappa}$ and $j(\bar{\kappa}) = \kappa$. 
\end{enumerate}
\end{definition}
A cardinal $\kappa$ is $\kappa$-$\Pi^1_1$-subcompact if and only if $\kappa$ is weakly compact. Cardinals $\kappa$ which are $\kappa^{+}$-$\Pi^1_1$-subcompact are called \emph{$\Pi^2_1$-subcompact cardinals} in \cite{NeemanSteelSubcompact}.

The next lemma characterizes $\lambda$-$\Pi^1_1$-subcompact cardinals in term of elementary embeddings with a fixed critical point. First, let us consider a definition due to Schanker \cite{Schanker2013}.
\begin{definition}[Schanker]
A cardinal $\kappa$ is \emph{$\theta$-nearly supercompact} if for all $A \subseteq \theta$, there is a transitive model $M$ of $\ZFC^-$, such that $A, \theta, \kappa \in M$, $M^{<\kappa} \subseteq M$ and there is an elementary embedding $j \colon M \to N$, $N$ is transitive, $\crit j = \kappa$, and $j\image \theta \in N$, and $\theta < j(\kappa)$.
\end{definition}
Schanker was interested in the case in which $\theta$ is small (relative to $2^\kappa$). In this case, he proved that a $\theta$-nearly supercompact cardinal might not be even measurable (see also \cite{BGHS2015}). 
The following lemma is implicit in \cite{Cody2020}. For the completeness of this paper, we provide a proof.

\begin{lemma}\label{lemma: subcompact using elementary embedding}
The following are equivalent for $\kappa < \lambda$ regular:
\begin{enumerate}
\item $\kappa$ is $\lambda$-$\Pi_1^1$-subcompact.
\item For every transitive model $M$ of size $|H(\lambda)|$, such that ${}^{<\kappa}M\subseteq M$ there is a transitive model $N$, ${}^{<\kappa}N\subseteq N$ and an elementary embedding $j\colon M \to N$ with critical point $\kappa$ such that $j\image M \in N$ and $\lambda < j(\kappa)$.
\item $\kappa$ is $|H(\lambda)|$-nearly supercompact.
\end{enumerate}
\end{lemma}
\begin{proof}
$(2) \implies (3)$ is clear, as the witnessing models for nearly supercompactness can be of minimal size. The implication $(3) \implies (2)$ follows from the fact that for $\theta = |H(\lambda)|$, $\theta^{<\kappa} = \theta$. Then, using Hauser's trick, \cite{Hauser91}, one can obtain an elementary embedding that respects a bijection between $M$ and $\theta$. Namely, let $M$ be as in (2) and let $\bar{M}$ be a model with the same universe as $M$ and an additional functions symbol $f$ which we interpret as a bijection between $\theta$ and $M$. Note that $M$ satisfies the assertion that for every $x$ of size $<\kappa$, there is $y$, such that $y = f \image x$. Apply the elementary embedding and using the nearly supercompactness hypothesis, we obtain a model $\bar{N}$ with a function symbol that we denote by $j(f)$, and moreover $j \image M = j(f)\image (j\image \theta) \in \bar{N}$. Reducing the language by removing the function symbol of $j$, we obtain the result. 

Let us show that (1) implies (2). Let $\kappa$ be $\lambda$-$\Pi^1_1$-subcompact and let $M$ witness that (2) is false. This means that $M$ is transitive, $M^{<\kappa} \subseteq M$, $|M| = |H(\lambda)|$ and for every transitive model $N$ and embedding $j \colon M \to N$, either $j$ is not elementary or $j\image M\notin N$. Since we may assume that the model $N$ has size $|H(\lambda)|$ (by taking an elementary substructure), this statement can be coded as a $\Pi^1_1$-statement on $H(\lambda)$, using some predicate $A$ in order to code the model $M$ and its elementary diagram. 

Applying (1), and Definition \ref{definition:pi11-subcompact}, we obtain cardinals $\bar{\kappa}$ and $\bar{\lambda}$ below $\kappa$ and a predicate $\bar{A}$ on $H(\bar{\lambda})$ that codes some transitive model $\bar{M}$. We also obtain an elementary embedding $\tilde{j} \colon \langle H(\lambda), \in \bar{A}\rangle \to \langle H(\lambda), \in, A\rangle$. By unwrapping the code $A$ for $M$, we conclude that $\tilde{j}$ codes an elementary embedding $j \colon \bar{M} \to M$. Recall that $\kappa$ is strongly inaccessible and $\bar\lambda < \kappa$, so $|H(\bar\lambda)| = |\bar{M}| < \kappa$. Since $M$ is closed under sequences of size $<\kappa$, $j\image \bar{M} \in M$. Let us take an elementary substructure of $M$ that contains $j\image \bar{M},\, \{j\image \bar{M}\}$ as elements of size $|H(\bar{\lambda})|$ and closed under $<\bar\kappa$-sequences. The transitive collapse of this model is coded by some subset of $H(\bar{\lambda})$ witnessing that the above $\Pi^1_1$-statement fails in $H(\bar{\lambda})$.  

Let us prove that (2) implies (1). Let us assume that (2) holds, and let us show that $\kappa$ is $\lambda$-$\Pi^1_1$-subcompact. Let $\Phi$ be a $\Pi^1_1$-statement that holds in the model $\langle H(\lambda), \in, A\rangle$. Applying the hypothesis, there is an elementary embedding with critical point $\kappa$ between some transitive model $M\supseteq H(\lambda) \cup \{A, H(\lambda)\}$ and a transitive model $N$ such that $j\image M \in N$ and $\lambda < j(\kappa)$. 

Since $N$ is closed under basic operations on sets, $j \image H(\lambda) = j(H(\lambda)) \cap j\image M \in N$. By taking the transitive collapse of $j \image H(\lambda)$ inside $N$, we conclude that $H(\lambda), A \in N$.

Working in $N$, the following hold: 
\[N \models \text{``}\langle H(\lambda), \in, A\rangle \models \Phi\text{''},\]
Let us note that the model $N$ does not contain all subsets of $H(\lambda)$, which means that in general the truth value of second order formulas would not be absolute between $N$ and $V$. So, the validity of the formula $\Phi^N$ uses the fact that $\Phi$ is a $\Pi^1_1$-statement.
In $N$, there is an elementary embedding $k = j\restriction H(\lambda)$ from the structure $\langle H(\lambda), \in ,A\rangle$ to $\langle j(H(\lambda)), \in, j(A)\rangle$. with critical point $\kappa$ and $k(\kappa) = j(\kappa) > \lambda$. 
Thus, by elementarity of $j$, the same holds in $M$: there are $\bar{\kappa}, \bar{\lambda} < \kappa$, $\bar{A} \subseteq H(\bar{\lambda})$ and an elementary embedding $\bar{k} \colon \langle H(\bar\lambda), \in, \bar{A}\rangle \to \langle H(\lambda), \in, A\rangle$. Moreover, we have $\langle H(\bar\lambda), \in, \bar{A}\rangle \models \Phi$.
\end{proof}

By starting with a more well-behaved model $M$, the $\lambda$-$\Pi^1_1$-subcompactness of $\kappa$ yields a better closure properties for the target model $N$, than what is stated in Lemma \ref{lemma: subcompact using elementary embedding}(2). 
\begin{lemma}\label{lemma:better closure for N}
Let $\kappa$ be $\lambda$-$\Pi^1_1$-subcompact. Let $M$ be a transitive model such that 
\begin{enumerate}
\item $|M| = |H(\lambda)|$ and $M$ is closed under $<\lambda$-sequences,
\item $P_\kappa\lambda, \lambda \in M$,
\item\label{requirement:M is a model of ZFC} $M$ is the transitive collapse of some elementary submodel of $H(\chi)$, for some $\chi$. 
\item $M$ has definable Skolem functions and 
\item $M$ contains a function symbol $g$ which is evaluated as a bijection between $\lambda$ and $M$. 
\end{enumerate}
Then, one can get $j \colon M \to N$ such that $N$ is closed under $<\lambda$-sequences, $\crit j = \kappa$ and $j \image M \in N$. 
\end{lemma}
\begin{proof}
Let $j\colon M \to N$ be as in Lemma \ref{lemma: subcompact using elementary embedding}. Since $\lambda \in M$, $j\image \lambda \in N$. 

Let us use the seed hull, as in \cite{Hamkins}, 
\[\mathbb{X}_{\{j\image \lambda\}} := \{j(f)(j\image \lambda) \mid f \in M,\, f \colon P_\kappa \lambda \to M\}.\]
Hamkins proves that $\mathbb{X}_{s} \prec N$ for every set of seeds $s$, and in particular $j \colon M \to \mathbb{X}_s$ is elementary. Hamkins' proof is done in the context of an elementary embedding from $V$ to some class, but it goes without change to our case, under the assumption that $M$ has definable Skolem functions.  

Let $\pi \colon \mathbb{X}_{\{j\image \lambda\}} \to N'$ be the transitive collapse, so $k = \pi \circ j \colon M \to N'$ is an elementary embedding. In order to show that $\crit k = \kappa$, let us notice that all ordinals up to $\lambda$ belong to $\mathbb{X}_{\{j\image \lambda\}}$, so $\crit \pi^{-1} \geq \lambda$. The main point is to verify that $N'$ is closed under $<\lambda$-sequences. Let $\{ y_i \mid i < i_\star\} \subseteq N'$, $i_\star < \lambda$. Then, by the definition of $N'$, for each $i$ there is a function $f_i \in M$ such that $y_i = \pi(j(f_i)(j\image \lambda))$. By the closure of $M$, $\vec{f} = \langle f_i \mid i < i_\star\rangle \in M$. 

Let us look at $k(\vec{f}) \in N'$. Since $k$ is elementary $N'$ satisfies enough set theory. Since $k \image \lambda = \pi(j\image \lambda) \in N'$, we conclude that 
\[k(\vec{f}) \image (k\image \lambda \cap k(i_\star)) = \{k(f_i) \mid i < i_\star\} \in N'.\]
Therefore, 
$A = \{k(f_i)(k\image \lambda) \mid i < i_\star\} \in N'$. Applying $\pi^{-1}$ and using the fact that $\crit \pi^{-1} \geq \lambda$, we get 
\[\pi^{-1}(A) = \{j(f_i)(j\image \lambda) \mid i < i_\star\},\]
Applying $\pi$ again the result follows.\end{proof}

Of course, Lemma \ref{lemma:better closure for N}(\ref{requirement:M is a model of ZFC}) can be replaced with the assertion that $M$ satisfies some weak version of set theory.

It is interesting to compare the relationship between Lemma \ref{lemma: localized strong compactness} and Lemma \ref{lemma: subcompact using elementary embedding} to the relationship between the strongly compact and the supercompact elementary embeddings. This comparison points to a possible normality assumption that should be added to the local $\mathcal{L}_{\kappa,\kappa}$-compactness characterization in order to get a model theoretical characterization of $\lambda$-$\Pi^1_1$-subcomapctness. 
Following Benda, \cite{Benda78}, we suggest to use type omission as a possible candidate for this additional hypothesis in the next section. 
\section{Type omission and \texorpdfstring{$\Pi^1_1$}{pi11}-subcompactness}\label{section: type omission}
We will use the following definition of a club, due to Jech \cite[Section 3]{Jech1972}.
\begin{definition}
Let $\kappa$ be a regular cardinal and let $X$ be a set. A set $C \subseteq P_\kappa X = \{x \subseteq X \mid |x| < \kappa\}$ is a club if
\begin{itemize}
\item for every $x \in P_\kappa X$ there is $y \in C$, $x \subseteq y$ and
\item for every increasing sequence $\langle x_i \mid i < i_\star\rangle$, $i_\star < \kappa$, $x_i \in C$, $\bigcup x_i \in C$. 
\end{itemize}
\end{definition}
By a theorem of Menas, every club contains a club of the form $C_F$ where $F \colon X \to P_\kappa X$ and $C_F = \{x \in P_\kappa X \mid \bigcup (F \image x) \subseteq x\}$, see \cite[Proposition 4.6]{JechHandbook}.
 
\begin{definition}
Let $\kappa \leq \lambda$ be cardinals and let $\mathcal{L}$ be a logic extending first order logic.  We say that \emph{$\kappa$-$\mathcal{L}$-compactness with type omission for languages of size $\lambda$} holds if for every $\mathcal{L}$-theory $T$ and $\mathcal{L}$-type $p$ such that for club many $T' \cup p'\in P_\kappa (T \cup p)$ there is a model of $T'$ that omits $p'$, then there is a model that realizes the theory $T$ and omits the type $p$.
\end{definition}
We remark that omitting larger types is easier while realizing larger theories is more difficult. In particular, any omitable type has a non-omitable subtype (e.g., the empty subtype is non-omitable). Thus, the restriction of the pairs of sub-theory and sub-type to some club is somewhat natural.

Benda prove \cite{Benda78} that compactness of type omission over $\mathcal{L}_{\kappa,\kappa}$ over arbitrary languages is equivalent to supercompactness. We give a different argument that provides a local equivalence and use only first order types and theories, with no infinitary quantifiers and connectors. 

\begin{theorem}\label{theorem: type omission and subcompactness}
Let $\kappa \leq \lambda = \lambda^{<\kappa}$ be cardinals, $\kappa$ regular and uncountable. The following are equivalent.
\begin{enumerate}
\item For every transitive model $M$ of size $\lambda$, ${}^{<\kappa}M\subseteq M$, there is a transitive model $N$ and an elementary embedding $j\colon M \to N$ such that $\crit j = \kappa$, $\lambda < j(\kappa)$ and $j \image M \in N$.
\item  $\kappa$-$\mathcal{L}_{\kappa,\kappa}$-compactness with type omission for languages of size $\lambda$ holds.
\item $\kappa$-$\mathcal{L}_{\omega,\omega}$-compactness with type omission for languages of size $\lambda$ holds.\footnote{The above formulation was obtained in response to a private communication with Will Boney. In the previous version of this paper, there was an additional well foundedness hypothesis.}
\end{enumerate}
\end{theorem}
\begin{proof}
Clearly $(2)\implies (3)$. 

Let us show that $(1)\implies (2)$. Let $T, p$ be as in the assumptions of $(2)$. Let $C$ be a club in $P_\kappa (T \cup p)$ such that for every $T' \cup p' \in C$ there is a model for $T'$ that omits $p'$. Let us apply Lemma \ref{lemma: subcompact using elementary embedding} for some $<\kappa$-closed transitive model $M \prec H(\chi)$ for $\chi$ sufficiently large, $|M| = \lambda$ and $T, p, C \in M$. By applying Menas' lemma in $M$, we obtain a function $F \in M$, $F\colon T \cup p \to P_{\kappa}(T \cup p)$ such that $C_F \subseteq C$. Using the hypothesis, we obtain an elementary embedding $j \colon M \to N$ where $N$ is transitive and $j\image M \in N$. Thus, $j(T) \cap j\image M = j\image T,\ j(p) \cap j\image M = j\image p, j(C) \cap j\image M = j \image C$ are in $N$.  Since $X = (j\image T) \cup (j\image p) = \bigcup j\image C \in N$, $|C| = \lambda < j(\kappa)$, and $X$ is closed under $j(F)$, we conclude that $(j\image T) \cup (j\image p) \in j(C_F) \subseteq j(C)$. 

So, in $N$ there is a model $\mathcal{A}$ for the theory $j\image T$ that omits the type $j\image p$. Although the language for the theory and the type is the value under $j$ of the original language and might contain more symbols, the symbols that appear in $j\image T$ and $j\image p$ are only the $j$-images of the original symbols. Therefore, by applying $j^{-1}$ on those symbols we conclude that $\mathcal{A}$ is isomorphic to a model for $T$ that omits $p$.

Let us now consider $(3)\implies (1)$. Let $M$ be a transitive model of size $\lambda$ which is closed under $<\kappa$-sequences. We would like to find an elementary embedding with critical point $\kappa$ and a model $N$ such that $j\image M \in N$ and $\lambda < j(\kappa)$. Similarly to the proof of Lemma \ref{lemma: localized strong compactness}, we define a language that contains for every $x\in M$ a constant $c_x$ as well as two additional constants $d, s$. We intend $d$ to be the critical point $\kappa$ and $s$ to be the set $j \image M$. 

The theory $T$ contains the statement ``$d$ is an ordinal below $c_\kappa$'', and the statements ``$c_\alpha \in d$'' for all $\alpha < \kappa$. We also include in $T$ the assertions ``$c_x \in s$'' for all $x\in M$ and ``$|s| < c_\kappa$'' (namely, that there is an injection from $s$ to a bounded ordinal below $c_\kappa$). Finally, we include in $T$ 
the full $\mathcal{L}_{\omega,\omega}$-elementary diagram of $M$.

We would like also to define a type that will be omitted. There are two offending objects that we would like to omit from our model: either witnesses for $s \neq j\image M$ or critical points below $\kappa$. The type $p$ is going to handle both cases. $p(x)$ is going to be the type of an element which is either in $s$ but not $c_z$ for any $z\in M$, or below $d$ but not in $s$. Namely, 
\[p(x) = \{\text{``} x\in s \cup d \text{''}\} \cup \{\text{``} x \neq c_z\text{''} \mid z \in M\}.\]

We would like to show that indeed on a club in $P_{\kappa} (T \cup p)$, there is a model for the sub-theory that omits the sub-type. Let $\theta$ be a sufficiently large regular cardinal, $M, T, p\in H(\theta)$.  

Let us fix a well order of $H(\theta)$, $\leq_\theta$. Pick some enumeration $e$ of $T$ and $p$ and let $C'$ be the club of all elementary substructures $Y \prec \langle H(\theta), \in, \leq_\theta\rangle$ that contains $M, T, p, e$ and satisfy $Y \cap \kappa \in \kappa$. Let $\Hull$ denote the Skolem hull function in the structure $\langle H(\theta), \in, \leq_\theta, M, T, p, e\rangle$, defined using the well order $\leq_\theta$.  So $Y \in C'$ iff $Y = \Hull(A)$ for some $A$, $|A| < \kappa$ and $Y \cap \kappa \in \kappa$.

Let us assume moreover that the map $x \mapsto c_x$ is definable in $H(\theta)$. Each such model $Y$ would be closed under sub-formulas, so if a formula $\varphi$ in $T$ or in $p$ contains the constant $c_x$ as a sub-formula and $e(\varphi) \in Y$ then $x \in Y$. 
Thus, for any element $x \in M$ the constant $c_x$ appears as a sub-formula of a formula in $Y \cap (T \cup p)$, if and only if $x$ belongs to $Y \cap M$.  

Let $T' \cup p' \in C$ iff $T' \cup p' \in P_\kappa(T' \cup p')$, $\Hull(T' \cup p') = Y \in C'$, and $Y \cap (T \cup p) = T' \cup p'$. One can easily describe a function $f \colon (T \cup p)^{<\omega} \to P_{\kappa}(T \cup p)$ such that $C$ consists of all elements which are closed under $f$, so in particular $C$ is a club.

Let us consider $T' \cup p'$ in $C$ and let $X \prec M$ be a corresponding elementary submodel, $X = \Hull(T' \cup p') \cap M$. We claim that the model $M$ itself with the evaluations $c^M_a = a$ for every $a\in X$, $d^M = X \cap \kappa$ and $s^M = X$ realizes $T'$ while omitting $p'$. First, $d^M$ is an ordinal below $c^M_\kappa$. Moreover, by the closure assumption, if a constant $c_a$ appears in $T'$ then $a\in X$. In particular, this model satisfies that whenever $c_a\in s$ appears in $T'$ then $a\in X$. Since $M$ is $<\kappa$ closed, there is some bijection between $X$ and an ordinal below $\kappa$ in $M$. The other assertions in $T'$ follow similarly. 

Let us consider $p'$. If $M$ does not omit $p'$ then there is some element $x \in M$ such that $x\in X \cup d^M$, $x \neq c_z^M$ for every $z$ such that the formula ``$x \neq c_z$'' belongs to $p'$. By the definition of $d^M$, $d^M \subseteq X$. By the closure of $Y$  ``$x \neq c_z$'' appears in $p'$ if and only if $z\in X$, which is what we need.

Now, we may apply the hypothesis of the lemma and obtain a model $N$ for $T$ that omits $p$. As in the proof of Lemma \ref{lemma: localized strong compactness}, the embedding $j\colon M \to N$ which is defined by $j(z) = c_z^N$ is an elementary embedding with critical point $\kappa$. By the type omission, $s^N = j \image M$. 

We are not done yet, since $N$ might be ill founded.\footnote{In the previous version of this paper, we added the hypothesis that the theory $T$ includes the $\mathcal{L}_{\omega_1,\omega_1}$-formula stating that the membership relation is well-founded.}  

\begin{claim}
Let $\kappa \leq \lambda$ be cardinals. Let us assume that for any transitive $M$ with $M^{<\kappa} \subseteq M$ and $|M| \leq \lambda$, there is a model $N$ and an elementary embedding $j\colon M \to N$, with $j\image M \in N$, $\crit j = \kappa$, $j(\kappa) > \lambda$.  
Then, for any such $M$, we can get the same conclusion with $N$ being transitive.
\end{claim}
\begin{proof}
Let $M$ be a model, satisfying the hypothesis of the claim. Let $M' \prec H(\chi)$ ($\chi$ large enough) be a larger model (so $M \subseteq M'$), closed under $<\kappa$-sequences, such that $M \in M'$, $\lambda + 1 \subseteq M'$ and $|M'| = \lambda$. Let $\bar{M}'$ be the transitive collapse of $M'$. Since $M$ is transitive, $M \in \bar{M}'$.

Let $j \colon \bar{M}'\to N'$ be as in the hypothesis of the claim. Note that $j \restriction M \colon M \to j(M)$ is a member of $N'$, as the intersection of $j \image \bar{M}'$ and $j(M)$. 

First, note that $j \image \lambda = j(\lambda) \cap j\image \bar{M}'\in N'$. Since $N'$ can compute the transitive collapse of $j\image \lambda$, we conclude that $\mathrm{Ord}^{N'} \supseteq \lambda + 1$. Similarly, since ${}^{<\kappa}\lambda \subseteq \bar{M}'$, we have ${}^{<\kappa}\lambda \in N'$. Since $j \image M \in N'$, we conclude that $j\image \lambda^{<\kappa} \in N'$ as well. 

Let $F \colon \lambda^{<\omega} \to M$ be a function such that $F \restriction \lambda$ is a bijection between $\lambda$ and $M$, and $F$ codes the Skolem functions of $M$, so for every $a \subseteq \lambda$ non-empty, $F \image [a^{<\omega}] \prec M$.  

Since $M$ is closed under $<\kappa$ sequences and $|M|^{N'} = \lambda < j(\kappa)$, we conclude that $j\image M \in j(M)$ and in particular there is some $\delta < j(\lambda)$ such $j(F)(\delta) = j\image M$.
Let us consider the following subset of $N'$, 
\[\tilde{N} = j(F) \image (j \image (\lambda^{<\kappa}) \cup \{\delta\}).\]
By the properties of $F$, $j\image M \subseteq \tilde{N}$ and $N' \models \tilde{N} \prec j(M)$. Since $N' \in V$, $V \models \tilde{N} \prec j(M)$ (see \cite{HamkinsYang}). 

Let us claim that $\tilde{N}$ is well-founded. If not, then there is an $\omega$-sequence of elements $a_n \in \tilde{N}$ such that $a_{n+1} \in^{N'} a_n$. Each $a_n$ is of the form $F(j(b_n), \delta)$ where $b \in \lambda^{<\kappa}$. Using the regularity of $\kappa$, there is $c \in j\image \lambda^{<\kappa} \subseteq N'$ such that $c$ codes the $\omega$-sequence $\langle j(b_n) \mid n < \omega\rangle$. So, this sequence is a member of $N'$. But this is absurd, as  this would imply that the sequence $\langle a_n\mid n <\omega\rangle$ is a member of $N'$, violating the fact that $N'\models$ Axiom of Foundation.  

So, we conclude that $\tilde{N}$ is well founded and for every $x \in M$, $j(x) \in \tilde{N}$. Let $\pi \colon \tilde{N} \to N$ be the transitive collapse and let $k \colon M \to N$ be $k(x) = \pi(j(x))$ for $x \in M$. It is straight-forward to verify that $k$ is an elementary embedding, $\crit k = \kappa$ and $k(\kappa) > \lambda$. Moreover, $k \image M = \pi(j\image M) \in N$. 
\end{proof}
This concludes the proof of Theorem \ref{theorem: type omission and subcompactness}.
\end{proof}

\begin{corollary}
For $\kappa \leq \lambda$, $\kappa$ is $\lambda$-$\Pi^1_1$-subcompact if and only if compactness for $\mathcal{L}_{\kappa,\kappa}$ with type omission holds for languages of size $|H(\lambda)|$.
\end{corollary}
Quantifying $\lambda$ out, we obtain a characterization for supercompactness. The first equivalence is due to Benda:
\begin{corollary}
The following are equivalent:
\begin{itemize}
\item $\kappa$ is supercompact. 
\item $\kappa$-$\mathcal{L}_{\kappa,\kappa}$-compactness with type omission. 
\item $\kappa$-$\mathcal{L}_{\omega,\omega}$-compactness with type omission.
\end{itemize}
\end{corollary}

\section{Ladder systems and trees}\label{section: ladder systems}
The following concept, isolated by Jech \cite{Jech1972} (under the name $(\kappa,\lambda)$-mess), generalizes the notion o a $\kappa$-tree two two-cardinal context and is suitable for the investigation of strongly compact and supercompact cardinals.


Recall, that for a cardinal $\rho$, we denote by $P_\rho X$ the set of all subsets of $X$ of size $<\rho$. In particular, for $\rho' < \rho$, $P_{\rho'}X \subseteq P_{\rho}X$.

Note that we do not assume that $\rho$ is regular. A set $E\subset P_\rho X$ is a club if for every $x \in P_\rho X$ there is $y \in E$ such that $x \subseteq y$ and $E$ is closed under increasing unions of length $<\cf \rho$. 
\begin{definition}[Jech]
A $P_\kappa\lambda$-tree is a function $\mathcal{T} = \langle \mathcal{T}_x \mid x \in P_\kappa \lambda\rangle$ such that:
\begin{itemize}
\item For every $x\in P_\kappa \lambda$, $\mathcal{T}_x \subseteq {}^x 2$, non-empty.
\item For $x \subseteq y \in P_\kappa \lambda$ and $\eta \in \mathcal{T}_y$, $\eta \restriction x \in \mathcal{T}_x$. 
\item For every $x\in P_\kappa \lambda$, $|\mathcal{T}_x| < \kappa$.
\end{itemize}
\end{definition}
We call $\mathcal{T}_x$ the $x$-th level of $\mathcal{T}$. A \emph{branch} through $\mathcal{T}$ is a function $\eta \colon \lambda\to 2$ such that $\eta\restriction x \in \mathcal{T}_x$ for all $x\in P_\kappa\lambda$. If $\kappa$ is inaccessible, then the third requirement holds trivially. 
In \cite[Section 2]{Jech1972}, Jech showed that $\kappa$ is strongly compact if and only if every $P_\kappa\lambda$-tree has a branch and $\kappa$ is inaccessible. Removing the inaccessibility assumption, this property is called the \emph{Strong Tree Property}, and it is known to consistently hold at accessible cardinals, see for example, \cite{VialeWeiss2011, Fontanella2014}. 

\begin{definition}\label{definition:ladder system}
Let $\mathcal T$ be a $P_\kappa\lambda$ tree. A set $L$ is a \emph{ladder system} on $\mathcal T$ if the following holds:
\begin{enumerate}
\item $L \subseteq \bigcup_{x\in P_\kappa\lambda} \mathcal T_x$.
\item\label{condition:club many} For club many levels $x$, $L \cap \mathcal T_x \neq \emptyset$. 
\item\label{condition:downwards clubs} If $\eta \in L \cap \mathcal T_x$ and $\cf (|x \cap \kappa|) > \omega$ then there is a club $E_\eta \subseteq P_{|x \cap \kappa|} x$ such that $\{\eta \restriction y \mid y \in E_\eta\} \subseteq L$.
\end{enumerate}

Let $\rho \leq \kappa$ be a regular cardinal.

A cofinal branch $b$ through $\mathcal T$ meets the ladder system $L$ $\rho$-\emph{cofinally} if for every $x \in P_\rho\lambda$ there is $z \supseteq x$ such that $b\restriction z \in L$.

A cofinal branch $b$ through $\mathcal T$ meets the ladder system $L$ $\rho$-\emph{club often} if for club many $x \in P_\rho\lambda$, $b\restriction x \in L$.
\end{definition}

Intuitively, a ladder system consists of a collection of ``good nodes'' in the tree which we would like the branch to go through, similarly to the \emph{Ineffable Tree Property} (ITP), \cite{Weiss}. Unlike ITP, we weaken our requirement by making sure that the set of good nodes is very rich --- below any node in a level of uncountable cofinality (in some sense) there are club many restrictions which are good as well.

\begin{definition}\label{def:ladder system catching}
Let $\rho \leq \kappa < \mu$ be cardinals. We say that \emph{ladder system catching property at $\rho$-clubs} (\emph{at $\rho$-cofinal sets}) for $P_\kappa\mu$ trees holds, if for every $P_\kappa \mu$-tree $\mathcal T$ and a ladder system $L$ there is a cofinal branch $b$ of $\mathcal{T}$ that meets the ladder system $L$ on a $\rho$-club (on a $\rho$-cofinal set).

We abbreviate this property by $\LSCP(P_\kappa\lambda, \rho\text{-clubs})$ or $\LSCP(P_\kappa\lambda,\rho\text{-cofinally})$, respectively. 
\end{definition}


\begin{theorem}\label{theorem: trees with ladders to subcompactness}
Let $\kappa \leq \mu= \mu^{<\kappa}$ be cardinals, $\kappa$ is inaccessible. The following are equivalent:

\begin{enumerate}
\item $\kappa$-$\mathcal{L}_{\omega,\omega}$-compactness with type omission for languages of size $\mu$.
\item $\LSCP(P_\kappa\mu, \kappa\text{-clubs})$. 
\item $\LSCP(P_\kappa\mu, \kappa\text{-cofinal})$. 
\end{enumerate}
\end{theorem}
\begin{proof}
$(2)\implies (3)$ is trivial.

Let us show that $(3) \implies (1)$. Recall that $(1)$ is the statement: for every first order theory $T$ over a language of size $\mu$ and a type $p(x)$, if for club many $T' \cup p' \in P_{\kappa}(T \cup p)$ there is a model $M$ that satisfies $T'$ and omits $p'$ then there is a model of $T$ that omits $p$. 

Indeed, let us assume that $T$ is a first order theory over a relational language $\mathcal{L}$ with $\mu$ many symbols and $p$ is a type. We will assume that $T$ is Henkinized (for every sentence of the form $\psi := \exists x \varphi(x, r)$, where $r$ is a sequence of constants, there is a constant $c_{\varphi, r}$ such that $c$ is a witness to the formula $\psi$ if and only if $\psi$ holds). 
So, finding a model for $T$ that omits the type $p$ is the same as finding a consistent complete extension, $\tilde{T}$, in which for every constant $c$ there is $\phi(x) \in p(x)$ such that $\neg \phi(c) \in \tilde{T}$. 
Let us assume that there are club many $T'\cup p'$ such that there is a model $M$ that realizes $T'$ and omits $p'$. Let us construct a tree $\mathcal{T}$ as follows. Pick an enumeration $e$ of length $\mu$ of all terms and formulas in the language $\mathcal{L}$. 

For every $a \in P_\kappa \mu$, $\eta \in \mathcal{T}_a$ if and only if there is a model $M_\eta$ such that 
\begin{enumerate}
\item for every $\alpha \in a$, if $e(\alpha)$ is a sentence then $\eta(\alpha) = 1 \iff M \models e(\alpha)$, and 
\item for every $e(\alpha) \in T$, $\eta(\alpha) = 1$.
\end{enumerate}

Since we assumed that the language $\mathcal{L}$ is Henkenized, for every $a \subseteq b \in P_\kappa \mu$, and $\eta \in \mathcal{T}_b$, the function $\eta \restriction a$ defines a sub-model of the model $M_\eta$, assuming that $e\image a$ is closed under sub-formulas and sub-terms. So, in this case we say that $\eta \restriction a$ omits $p \cap e\image a$ if there is no constant $c \in e \image a$ such that for every $\varphi(x) \in p$, $\eta(e^{-1}(\varphi(c))) = 1$.   

We are now ready to construct the ladder system $L$.
Let $a \in P_\kappa \mu$, such that the collection of formulas in $e \image a$ is closed under sub-formulas and apply substitution of a variable with a term in a formula.
Let us define $\eta \in L \cap \mathcal{T}_a$ if there is a model $M_\eta$ that omits $e \image a \cap p$. Note that if $a \cap \kappa$ is of uncountable cofinality, then there are club many $b \in P_{|a \cap \kappa|} a$ such that $\eta \restriction b$ is an assignment of a Henkenized theory, $e\image b$ is closed under sub-formulas and omits $e \image b \cap p$. Indeed, in order to omit the sub-types of $p$, one needs to verify that for every constant symbol $c$ in $e \image b$, there is a formula $\varphi \in p \cap e \image b$ such that $\varphi(c) \in e \image b$ and $\eta(e^{-1}(\varphi(c)) = 0$. Since this is true for $\eta$ and $a$, we can define a function sending $c \in e \image a$ so $\varphi \in p \cap e \image a$ such that $\varphi(c) \in e \image a$ and $M_{\eta} \models \neg \varphi(c)$. Thus, for any $b \subseteq a$ which is closed under this function, would omitted the subtype $p \cap e\image b$. 

Let $b$ be a cofinal branch though the tree $\mathcal T$, and assume that $b$ meets $L$ cofinally. Since $b$ is a cofinal branch, it defines a complete theory extending $T$ and thus a model of $T$. Let us call this model $M_b$. We want to verify that the type $p$ is omitted. Indeed, let $z \in M_b$. Let $x\in P_\kappa \lambda$ contain the ordinal in which the constant for $z$ is enumerated. Let $y\supseteq x$ such that there is $\eta' \in L \cap\ \mathcal{T}_y$, $\eta' = b \restriction y$. Since $\eta'$ represents a model that omits a sub-type of $p$ and contains the constant $z$, there must be a formula $\varphi \in p \cap e\image y$ such that $\beta = e(\neg\varphi(z)) \in \dom \eta'$ and $\eta'(\beta) = 1$.  Thus, $z$ does not realize $p$. 

Let us finally show $(1) \implies (2)$. Let $M$ be a transitive model of size $\mu$ containing $T$ and $L$, and closed under $<\kappa$-sequences. 
By Theorem \ref{theorem: type omission and subcompactness} and the hypothesis, there is an elementary embedding $j\colon M \to N$, with critical point $\kappa$, $j\image \mu\in N$. Let $D \subseteq P_\kappa\mu$ be a club such that for all $x \in D$, $L \cap \mathcal{T}_x \neq \emptyset$ and belongs to $M$. Then $j\image \mu\in j(D)$ and in particular, there is some $\tilde{\eta} \in j(\mathcal{T})_{j\image \mu}\cap j(L)$. 

Let $b$ be the following branch: 
\[b(x) = j^{-1} \image (\tilde{\eta} \restriction j\image x) = \{(\zeta, \epsilon) \mid b(j(\zeta)) = \epsilon\}.\] 
Buck in $M$, let $E_\eta$ be the club, as in Definition \ref{definition:ladder system}, (3). Let us apply $j$ to the function $\eta \mapsto E_\eta$ and let $\tilde{E}$ be the obtained club. So $\tilde{E}$ is a club on $P_\kappa j\image\mu$ (since $\kappa = |j\image \mu\cap j(\kappa)|$). For every $z \in P_\kappa j\image\mu$, $z = j(w)$ for some $w \in P_\kappa \mu$, so $D = j^{-1} \image (j(E)_{\tilde{\eta}})$ is a club in $P_\kappa \mu$. For all $x \in D$, 
$b \restriction x \in L$, (as $j(x) \in E_\eta$ and $j(b \restriction x) = \eta \restriction j(x) \in j(L)$), as wanted.
\end{proof}
Again, by quantifying $\mu$ out, we obtain a characterization for supercompactness: 
\begin{corollary}
The following are equivalent for an inaccessible cardinal $\kappa$:
\begin{enumerate}
\item $\kappa$ is supercompact.
\item For every regular $\mu \geq \kappa$, $\LSCP(P_\kappa\mu, \kappa\text{-cofinally})$.
\end{enumerate}
\end{corollary}
\section{Down to \texorpdfstring{$\omega_2$}{aleph2}}\label{section: aleph2}
In the previous sections, the inaccessibility of $\kappa$ played a major role. We might ask whether meeting ladder systems cofinally or club often is still equivalent at accessible cardinals. We will focus on the case of $\omega_2$. In this case, we will refer to Definition \ref{def:ladder system catching} with $\rho = \omega_1$ or $\omega_2$ , which means that even in the case of $P_{\omega_2}\omega_2$-trees, which are typically identified with $\omega_2$-trees, we will need to consider their structure with respect to $P_{\omega_1}\omega_2$ as well.


For $\lambda = \omega_2$, the ordinals are a club in $P_{\omega_2} \lambda$. Nevertheless, for each ordinal $\alpha > \omega_1$, the ordinals below $\alpha$ are not a club in $P_{\omega_1} \alpha$. This means that even in this special case, we cannot treat the trees as simple $\omega_2$-trees but rather as $P_{\omega_2} \omega_2$-trees, where levels of countable size play an important role. This is a non-typical scenario, as restricting the tree and the ladder system to a club does not preserve the properties of the structure. 

\begin{theorem}\label{thm: tree property catching clubs}
It is consistent relative to a supercomapct cardinal, that for every $\mu \geq \omega_2$, $\LSCP(P_{\omega_2}\mu, \omega_1\text{-club})$. 
\end{theorem}
\begin{theorem}\label{thm: catching unbounded but not club}
It is consistent relative to a supercomapct cardinal, that for every $\mu \geq \omega_2$, 
$\LSCP(P_{\omega_2}\mu, \omega_1\text{-cofinally})$ but $\neg \LSCP(P_{\omega_2}\omega_2, \omega_1\text{-club})$.
\end{theorem}

For the first theorem, we will use the standard Mitchell forcing.
\begin{definition}[Mitchell, \cite{Mitchell}]
Let $\kappa$ be an inaccessible cardinal. The Mitchell poset $\mathbb{M}(\kappa)$ consists of conditions $p = \langle a, m\rangle$ where $a \in \Add(\omega,\kappa)$ and $m$ is a partial function with countable support such that for any $\alpha \in \supp m$, $\Vdash_{\Add(\omega,\alpha)} m(\alpha) \in \dot{\Add}(\omega_1, 1)$. 

We order the conditions of the forcing by $\langle a, m\rangle \leq \langle a', m'\rangle$ if $a \leq a'$ in the Cohen forcing $\Add(\omega,\kappa)$, $\dom m \supseteq \dom m'$ and $a \restriction \alpha \Vdash_{\Add(\omega,\alpha)} m(\alpha) \leq_{\dot{\Add}(\omega_1,1)} m'(\alpha)$ for every $\alpha \in \dom m'$.
\end{definition}
In \cite{Mitchell}, Mitchell showed that if $\kappa$ is weakly compact then $\mathbb{M}(\kappa)$ forces that the tree property holds at $\kappa$. Starting with a stronger large cardinal hypothesis, the Mitchell poset can be used to obtain the strong tree property, the ineffable tree property and more (see for example, \cite{Weiss}).

To establish Theorem \ref{thm: tree property catching clubs}, it is enough to proof the following.
\begin{lemma}
Let $\kappa$ be $\lambda$-$\Pi^1_1$-subcompact. Then in the generic extension by $\mathbb{M}(\kappa)$, $\LSCP(P_{\omega_2}\lambda, \omega_1\text{-club})$ holds.
\end{lemma}
\begin{proof}
Let us consider a name for a tree $\dot{T}$ and a ladder system $\dot{L}$ on $P_\kappa \lambda$ of the generic extension. By the $\kappa$-c.c.\ of $\mathbb{M}(\kappa)$, the set $\left(P_\kappa\lambda\right)^V$ is unbounded in $\left(P_\kappa\lambda\right)^{V[G]}$. Moreover, one can easily code all names for elements in $P_\kappa\lambda$, $\dot{T}$ and $\dot{L}$ into a transitive structure $M$ of size $\lambda$. We will assume that $M$ satisfies some portion of $\mathrm{ZFC}$, and in particular it satisfies choice and the basic theory of forcing (including the forcing theorem for $\Sigma_n$ formulas, where $n$ is sufficiently large).

By Lemma \ref{lemma: subcompact using elementary embedding}, there is an elementary embedding 
\[j\colon M \to N,\]
such that $j \image M \in N$. We would like to lift this embedding to an elementary embedding from $M[G]$ to $N[H]$, where $G$ is a $V$-generic filter for $\mathbb{M}(\kappa)$ and $H$ is an $N$-generic filter for $j(\mathbb{M}(\kappa))$. We cannot construct $H$ in $V[G]$, so in order to construct $H$ we force with $j(\mathbb{M}(\kappa)) / \mathbb{M}(\kappa)$ over $V[G]$. 

Indeed, it is obvious that $\mathbb{M}(\kappa) = j(\mathbb{M}(\kappa)) \restriction \kappa$. Moreover, since for every $p \in \mathbb{M}(\kappa)$, $j(p) = p$, we conclude that for a generic filter $H \subseteq j(\mathbb{M}(\kappa))$, letting $G = H \restriction \kappa$, the embedding $j$ can be extended to an elementary embedding $j^\star \colon M[G] \to N[H]$. 

As in Theorem \ref{theorem: trees with ladders to subcompactness}, by taking an element $\eta \in j(\dot{T})^{H}_{j\image \lambda} \cap j(\dot{L})^H$, we obtain a branch through $\dot{T}^G$, 
\[b = \{ j^{-1}(\eta \restriction j\image x) \mid x \in \left(P_\kappa \lambda\right)^{V[G]}\}.\]
We would like to show that $b$ belongs to $V[G]$ and that it meets $\dot{L}^G$ on a club. 

The forcing $j(\mathbb{M}(\kappa)) / G$ cannot add new branches to a $P_\kappa \lambda$ trees (see, for example, \cite{Weiss}, or Claim \ref{claim: no new branch} ahead). Thus, $b\in V[G]$. Moreover, in $N[H] \subseteq V[H]$, there is a club in $P_{\omega_1}\lambda$ in which $b$ intersects $L$, since $\cf \kappa = \omega_1$ in the generic extension. We would like to claim that the same holds in $V[G]$. Assume otherwise and let us consider 
\[S = \{x \in P_{\omega_1} \lambda \mid b \restriction x \notin L\} \in V[G].\]
In $N[H] \subseteq V[H]$, $S$ is non-stationary. But the forcing $j(\mathbb{M}(\kappa)) / G$ is proper in $V[G]$ since it is a projection of a product of a $\sigma$-closed forcing and a c.c.c.\ forcing.
\end{proof}

In order to prove the Theorem \ref{thm: catching unbounded but not club}, we will modify Mitchell forcing in order to introduce at each inaccessible level a counterexample for the stronger property of ladder system catching at clubs, while still preserving the tree property. 

Recall that given an ordinal $\alpha \leq \omega_2$, and a binary tree $T \subseteq 2^{\leq \alpha}$. Let $\mathcal{T}$ be the $P_{\omega_2} \alpha$-tree defined by $\mathcal{T}_x = \{r \restriction x \mid r \in T_{\sup x}\}$.  

We will say that $L$ is a ladder system on an $\alpha$-tree $T$ if is a ladder system of the corresponding $P_{\omega_2}\alpha$-tree $\mathcal{T}$. 

\begin{definition}
Let $\alpha$ be a regular cardinal. We define a forcing notion $\mathbb{S}(\alpha)$ that introduces an $\alpha$-tree  $T$ with a ladder system $L$ and branches $\{b_t \mid t \in T\}$ by initial segments, as follows. 

A condition $p \in \mathbb{S}(\alpha)$ is of tuple $p = \langle t, \ell, b, f\rangle$ where:
\begin{enumerate}
\item $t \subseteq {}^{\leq\gamma}2$ is a normal binary tree of successor height, $\gamma + 1 < \alpha$.
\item $\ell$ is a function with a domain which is a closed subset of $\gamma + 1$, and for every $\xi \in \dom \ell$ of uncountable cofinality, there is a member $x \in t_\xi$ and a club $E_x$ at $P_{\omega_1}\xi$, such that $\ell(\xi) = \{x\}\cup \{x \restriction z \mid z \in E_x\}$.
\item $b$ is a function from $t$ to $t_\gamma$ such that $x \leq_t b(x)$ for all $x\in t$.
\item\label{item: definition of f}
$f$ is a set of pairs of the form $(z, r)$ where $z \in P_{\omega_1} \gamma$ and $r \in t_{\sup z} \cup \{-1\}$.  If $(z, r), (z', r') \in f$ are distinct then $\sup z \neq \sup z'$. The set 
\[\{\sup z \mid \exists r\neq -1,\, (z,r) \in f\}\]
is nowhere stationary. 
\item For every $\beta \in \dom \ell$, $\range f \cap \ell(\beta) = \emptyset$.  
\end{enumerate} 
We order $\mathbb{S}(\alpha)$ by $p = \langle t_p, \ell_p, b_p, f_p\rangle \leq q = \langle t_q, \ell_q, b_q, f_q\rangle$ if $t_p$ end extends $t_q$, $\ell_p$ end extends $\ell_q$ above the height of $t_p$ and for every $x\in \dom b_q$, $b_q(x) \leq_{t_p} b_p(x)$ and $f_p$ end extends $f_q$.
\end{definition}

The case $r = -1$ in item (\ref{item: definition of f}) is just a place holder for cases in which we want the ordinal $\sup x$ to be outside of the domain of the generic function. In this case, we abuse notation and declare the domain on $f$ at $\sup x$ to be empty.

Let us introduce the following notions which would be useful through the rest of the proof.
\begin{notation}\label{notation: generic tree}
If $S \subseteq \mathbb{S}(\alpha)$ is a generic filter, then:
\begin{itemize}
\item $T_\alpha = \bigcup \{t \mid \exists \langle t, \ell, b, f\rangle \in S\}$ is a binary $\alpha$-tree, 
\item For each $x \in T_\alpha$, let $B_\alpha(x) = \bigcup \{b(x) \mid \exists \langle t, \ell, b, f\rangle \in S,\, x \in \dom b\} \in {}^\alpha 2$ is a cofinal branch at $T_\alpha$.  
\item $L_\alpha = \bigcup \{\ell \mid \exists \langle t, \ell, b, f\rangle \in S\}$ is a ladder system on $T_\alpha$. 
\item $F_\alpha = \bigcup \{f \mid \exists \langle t, \ell, b, f\rangle \in S\}$. 
\end{itemize}
When $\alpha$ is clear from the context, we will omit it.

\end{notation}
The role of $F_\alpha$ is to kill potential branches that meet $L_\alpha$ on a club. Note that the set $\{\sup x \mid x \in \dom F_\alpha\}$ is non reflecting stationary subset of $\omega_2$. The $b$-components insure that the tree $T_\alpha$ would have many branches in the generic extension (otherwise, the plain tree property would fail). The existence of many branches given by $b$ is crucial in the proof of the strategic closure of the forcing. 

We refer the reader to \cite[Definitions 5.8, 5.15]{CummingsHandbook}, for the definition of $\sigma$-closed and $\alpha$-strategically closed forcings.
\begin{claim}
$\mathbb{S}(\alpha)$ is $\sigma$-closed, $\alpha$-strategically closed and of size $2^{<\alpha}$. 
\end{claim}
\begin{proof}
Let $\langle p_\xi\mid \xi< \epsilon\rangle$ be the game played so far, $\epsilon < \alpha$. We denote by $p_\xi = \langle t_\xi, \ell_\xi, b_\xi, f_\xi\rangle$, and we let $\delta_\xi = \max \dom \ell_\xi$ and $\gamma_\xi + 1$ be the height of the tree $t_\xi$. 

At successor steps, player Even does not move. At limit steps $\epsilon$, let us define $t_\epsilon$. If player Odd did not move co-boundedly below $\epsilon$, then player Even does need to do anything.

Otherwise, the conditions $p_\xi$ are strictly decreasing on an unbounded subset of $\epsilon$. Let us construct the condition which player Even would play. First, let us define $t_\epsilon$. This is a tree of height $(\sup_{\xi < \epsilon} \gamma_\xi) + 1$. Let $\tilde{t} = \bigcup_{\xi < \epsilon} t_\xi$. For each $x \in \tilde{t}$, let $B(x)$ be $\bigcup_{\xi_\star < \xi < \epsilon} b_\xi(x)$, where $\xi_\star$ is the level of $x$. We define \[t_{\epsilon} = \tilde{t} \cup \{B(x) \mid x \in \tilde{t}\}.\]
We let  
$b_\epsilon(x) = B(x)$ 
for $x \in \tilde{t}$ and $B(x) = x$ for nodes $x$ in the top level of $t_{\epsilon}$.

Let $\tilde{\ell} = \bigcup_{\xi < \epsilon} \ell_\xi$. 
If $\sup \delta_\xi < \sup \gamma_\xi = \gamma_\epsilon$, we let $\ell_\epsilon = \tilde{\ell}$. 
Otherwise, we need to define $\ell(\gamma_\epsilon)$. 
For $\cf \epsilon = \omega$, we can define $\ell(\gamma_\epsilon) = \emptyset$, and $f_\epsilon = \bigcup_{\xi < \epsilon} f_\xi \cup \{(x, -1)\}$ for some $x$ with $\sup x = \gamma_\epsilon$. 
If $\cf \epsilon > \omega$, we pick an arbitrary $x \in t_{\gamma_\epsilon}$, and let 
\[E_x = \{y \in P_{\omega_1} \gamma_\epsilon \mid \sup y \in \{\gamma_\xi \mid \xi < \epsilon\} \}.\]
We set $\ell(\gamma_\epsilon) = \{x\} \cup \{x \restriction z \mid z \in E_x\}$.

We need to verify that the definition works. Note that the only non-trivial requirement is the empty intersection of $\range f$ and $\ell(\beta)$ for all $\beta \in \dom f$. The requirement holds automatically for all $\beta \notin \{\gamma_\xi \mid \xi \leq \epsilon\}$. For $\beta = \gamma_\xi$, if $y \in \ell(\beta)$ then $\sup \dom y \in \{\gamma_\xi \mid \xi \leq \epsilon\}$, but for each such $y$, if $y \in \dom f$, then $f(y) = -1$. 

Since the strategy is trivial at finite steps, the forcing is $\sigma$-closed.
%
\end{proof}
\begin{claim}\label{claim:S is working}
Let $\alpha$ be a regular cardinal, $\alpha \geq \omega_2$. In the generic extension by $\mathbb{S}(\alpha)$ there is no branch of the generic tree $T_\alpha$ that meets the generic ladder system $L_\alpha$ on an $\omega_1$-club.
\end{claim}
\begin{proof}
Let $\dot{b}$ be a name for some a branch and let $\dot{C}$ be a name for a club. Let $p$ be a condition in $\mathbb{S}(\alpha)$. We want to find a condition $q \leq p$ such that $q \Vdash \check{a} \in \dot{C}$, $\dot{b} \restriction a = \check x$ and $\check x \notin \dot{L}$. 

Work inside some countable model $M$ such that $p, \dot{b}, \dot{C}, \mathbb{S}(\alpha) \in M$, and let $\delta = \sup(M \cap \alpha)$. By taking an $\omega$-sequence of extensions of $p$ inside $M$, we obtain an $M$-generic filter $G$. By the $\sigma$-closure of the forcing, there are many conditions $q$ such that $G = \{q' \in M \mid q' \geq q\}$. Any such condition is a lower bound for the filter $G$.  

Since $G$ is $M$-generic, for every $\zeta \in M$, the value of $\dot{b}(\zeta)$ is determined by some condition in $G$. Therefore, there some $x \colon M \cap \alpha \to 2$ such that $(\dot{b} \cap M)^G = x$. Note that for each condition $q$ as above, $q \Vdash \dot{b} \restriction (M \cap \alpha) = \check{x}$, and in particular for some $y$, $q \Vdash \check y \in \dot{T}_{\alpha}$ and $y \restriction (M \cap \alpha) = x$.

Since $\cf \delta = \omega$, we can pick $q = (t^q, \ell^q, b^q, f^q)$ to be a lower bound of the conditions in $G$, such that $f^q(M \cap \delta) = x$, and $\ell^q(\delta) = \emptyset$. This is possible, since the height of $t^q$ is at least $\delta + 1$, $y \in t^q_{\delta + 1}$ and $x = y \restriction (M \cap \delta)$. 

Since $q$ is $M$-generic and $\dot{C}\in M$ is forced to a club, $q \Vdash M \cap \alpha \in \dot{C}$. 
Finally, $q \Vdash \dot{b} \restriction (\check M \cap \check \alpha) = \check {x} \notin \dot{L}$.
\end{proof}

\begin{definition}
Work in the generic extension by $\mathbb{S}(\alpha)$. Let $\dot{\mathbb{T}}(\alpha)$ be the $\mathbb{S}(\alpha)$-name for the forcing that adds a club disjoint from the set $\{\sup x \mid x \in \dom F\}$, using bounded initial segments.
\end{definition}
The following observation is standard:
\begin{claim}
$\mathbb{S}(\alpha) \ast \dot{\mathbb{T}}(\alpha)$ contains an $\alpha$-closed dense subset.
\end{claim}
\begin{proof}
Let $D$ be the set of all conditions $\langle (t,\ell,b,f), q\rangle$ such that the height of $t$ is $\gamma + 1$, $\max \dom \ell = \max q = \gamma$. The set $D$ is dense and $\alpha$-closed.
\end{proof}

Next, we would like to define a variant of Mitchell's forcing, $\mathbb{M}'(\kappa)$. We define it by induction on $\rho\leq \kappa$. We verify throughout the inductive definition that there are natural projections from $\mathbb{M}'(\zeta)$ to $\mathbb{M}'(\rho)$ for $\rho < \zeta$, given by taking the restrictions of all the components in the condition. 

A condition in $\mathbb{M}'(\rho)$ is of the form $\langle a, m, s, t\rangle$ where:
\begin{enumerate}
\item $a \in \Add(\omega,\rho)$.
\item $m$ is a function with countable support (contained in $\rho$), such that for all $\alpha \in \supp m$, $\Vdash_{\Add(\omega,\alpha)} m(\alpha) \in \Add(\omega_1, 1)$. 
\item $s$ is a partial function with Easton support contained in the inaccessible cardinals $<\rho$, and for every $\alpha \in \dom s$, $\Vdash_{\mathbb{M}'(\alpha)} s(\alpha) \in \mathbb{S}(\alpha)$. 
\item $t$ is a partial function with Easton support contained in the inaccessible cardinals $<\rho$, and for every $\alpha \in \dom t$, $\Vdash_{\mathbb{M}'(\alpha) \ast \mathbb{S}(\alpha)} t(\alpha) \in \mathbb{T}(\alpha)$.
\end{enumerate}

We order the forcing naturally: $(a, m, s, t) \leq (a', m', s', t')$ iff $a \leq a'$, for all $\alpha < \rho$, $a' \restriction \alpha \Vdash m(\alpha) \leq m'(\alpha)$, $(a \restriction \alpha, m \restriction \alpha, s \restriction \alpha, t \restriction \alpha) \Vdash (s(\alpha), t(\alpha)) \leq (s'(\alpha), t'(\alpha))$.

We will force with $\mathbb{M}'(\kappa)\ast \mathbb{S}(\kappa)$, so the forcing at $\kappa$ behaves differently than the forcing at lower inaccessible cardinals: for each inaccessible $\alpha < \kappa$ we force with $\mathbb{S}(\alpha) \ast \dot{\mathbb{T}}(\alpha)$ while for $\kappa$ itself we just force with $\mathbb{S}(\kappa)$, without $\dot{\mathbb{T}}(\kappa)$. This strategy traces back to Kunen's proof \cite{KunenIdeals}, and appears in countless works where different compactness and anti-compactness principles are compared.

\begin{lemma}\label{lemma: catching unbounded but not clubs}
Let $\kappa$ be $\lambda$-$\Pi^1_1$-subcompact. Then in the generic extension by $\mathbb{M}'(\kappa) \ast \dot{\mathbb{S}}(\kappa)$, $\LSCP(P_{\omega_2}\lambda, \omega_1\text{-cofinal})$ holds but $\neg \LSCP(P_{\omega_2}\omega_2, \omega_1\text{-club})$. 
\end{lemma}
\begin{proof}
Since our forcing notion is of the form $\mathbb{M}'(\kappa)\ast \dot{\mathbb{S}(\kappa)}$, by Claim \ref{claim:S is working}, the generic tree and ladder system which is introduced by $\mathbb{S}(\kappa)$ would witness the failure of $\LSCP(P_{\omega_2}\omega_2, \omega_1\text{-club})$.

Let us turn now to showing that $\LSCP(P_{\omega_2}\lambda, \omega_1\text{-cofinal})$ holds. Let $\dot{\mathcal T}$ be a name for a $P_{\omega_2}\lambda$-tree and let $\dot{\mathcal L}$ be a name for a ladder system on $\dot{\mathcal T}$.

As in the proof of the previous case, we start with a transitive model $M$, which contains all relevant information and obtain from Lemma \ref{lemma: subcompact using elementary embedding} a transitive model $N$ and an elementary embedding, $j\colon M \to N$ with $j\image \lambda \in M$, $\lambda < j(\kappa)$. This time, we would like to require more closure from $N$, so we will assume that $M$ satisfies the further requirements of Lemma \ref{lemma:better closure for N}, and conclude that we can pick $N$ to be closed under $<\lambda$-sequences.

Let $G \ast S\subseteq \mathbb{M}'(\kappa) \ast \mathbb{S}(\kappa)$ be a generic filter. We would like to find a generic $G' \ast S' \subseteq j(\mathbb{M}'(\kappa) \ast \mathbb{S}(\kappa))$ and lift the embedding to en embedding $\tilde{j} \colon M[G][S] \to N[G'][S']$. 

Since the $f$-part which is introduced in the forcing $\mathbb{S}(\kappa)$ is a non-reflecting stationary set, there is no hope to lift this embedding without a forcing component that would add a club disjoint from it. So, $\tilde{j}$ exists only in a generic extension of $V[G][S]$.

First, let us show that
\[j(\mathbb{M}'(\kappa)) \cong \mathbb{M}'(\kappa) \ast \dot{\mathbb{S}}(\kappa) \ast \dot{\mathbb{T}}(\kappa) \ast \dot{\mathbb{Q}}.\]
Indeed, the map sending $(a, m, s, t) \in j(\mathbb{M}'(\kappa))$ to $(a \restriction \kappa, m \restriction \kappa, s \restriction \kappa + 1, t \restriction \kappa + 1)$ is a projection. Since $V_\kappa \subseteq N$, it is easy to verify that $\dot{\mathbb{S}}(\kappa), \dot{\mathbb{T}}(\kappa)$ and the forcing $\mathbb{M}'(\kappa)$ are computed in the same way in $M$ and in $N$, and therefore this map projects $j(\mathbb{M}'(\kappa))$ onto $\mathbb{M}'(\kappa) \ast \dot{\mathbb{S}}(\kappa) \ast \dot{\mathbb{T}}(\kappa)$. 
Let $T \subseteq \mathbb{T}^{G \ast S}(\kappa)$ be a $V[G][S]$-generic filter,
and let $\mathbb{Q}$ be the quotient forcing:
\[\mathbb{Q} := j(\mathbb{M}'(\kappa)) /\left(\mathbb{M}'(\kappa) \ast \dot{\mathbb{S}}(\kappa) \ast \dot{\mathbb{T}}(\kappa)\right) = j(\mathbb{M}'(\kappa)) / (G \ast S \ast T).\]






Let $C = \bigcup T$ be the generic club introduced by $\mathbb{T}(\kappa)$. In order to lift $j$, we must find a generic filter $G' \subseteq j(\mathbb{M}'(\kappa))$ and $S' \subseteq j(\mathbb{S}(\kappa))$ such that for every $p \in G \ast S$, $j(p) \in G' \ast S'$.
By the structure of the conditions in $\mathbb{M}'(\kappa)$, this implies that $G' \restriction \kappa = G$, and for every $s \in S$, $j(s)$ is in the generic $S'$ for $j(\mathbb{S}(\kappa))$. As usual, we choose $G'$ such that $G' \restriction \kappa + 1 = G \ast S \ast T$ and $G' / (G \ast S \ast T)$ is a generic filter for $\dot{\mathbb{Q}}^{G' \restriction \kappa + 1}$ over $M[G' \restriction \kappa + 1]$. 

We would like to find a master condition---a condition in $j(\mathbb{S}(\kappa))$, $m$ such that for all condition $s \in \mathbb{S}(\kappa)$ that appear in the generic filter $S$, $m \leq j(s)$. This would be sufficient as all conditions in the generic filter $G$ are unmoved by $j$. 

Let $T_\kappa, B_\kappa, L_\kappa, F_\kappa$ be the generic tree, branches, ladder system and function introduced by $\mathbb{S}(\kappa)$, respectively, as defined in Notation \ref{notation: generic tree} (do not confuse the generic tree $T_\kappa$ with the generic filter for the forcing $\mathbb{T}(\kappa)$, $T$).

Take $t_m = T_\kappa \cup \range B_\kappa \in 2^{\leq \kappa}$. So,  $t_m$ is a tree of height $\kappa + 1$. Let $\ell_m$ extend the generic ladder system $L_\kappa$ by adding one element in the level $\kappa$. Since $\kappa$ is forced to have uncountable cofinality in the generic extension by $j(\mathbb{M}'(\kappa))$, $\ell_m(\kappa)$ is obtained by picking one arbitrary element $\eta$ from the $\kappa$-th level of the tree and using the generic club $C$ that was introduced by $\mathbb{T}(\kappa)$: the club $E_\eta$ consists of all $x \in P_{\omega_1} \kappa$ such that $\sup x \in C$. 

Let $b_m = B_\kappa$, the collection of all generic branches. More precisely, for every $x \in T_\kappa$, we define $b_m(x)$ to be the node in $t_m$ which lie on top of the cofinal branch $B_\kappa(x)$, and $b_m(x) = x$ for $x \in t_m \cap {}^\kappa 2$. Let $f_m = F_\kappa$. The generic club $C$ witnesses the domain of $F$ to be non-stationary. Moreover, since $C$ does not intersect $\{\sup x \mid x \in \dom F\}$, we conclude that $\ell_m(\kappa)$ is disjoint from $F$.

Finally, we take a generic $S'$ such that $m\in S'$. By the above discussion, in $V[G'][S']$, the embedding $j$ lifts. Let us denote by $j^\star \colon M[G][S] \to N[G'][S']$ the lifted embedding. 

As in the proof of Theorem \ref{theorem: trees with ladders to subcompactness}, we obtain a branch $b$ by considering the value of the ladder system at $j\image \lambda$: The element $j \image \lambda$ is a member of the club which is included in the domain on $j^*(\mathcal L)$. We take $\eta \in j^*(\mathcal L)(j \image \lambda)$, and define 
\[b = \{j^{-1}(\eta \restriction j\image z) \mid z \in P_{\omega_2} \lambda\}.\]

We claim that $b\in V[G]$.  

\begin{claim}\label{claim: no new branch}
Assume that in $V[G]$, there is no cofinal branch in $\mathcal T$ that meets the ladder system $\mathcal L$ $\omega_1$-cofinally. Then, the forcing $\mathbb{T}(\kappa) \ast \mathbb{Q} \ast j(\mathbb{S}(\kappa))$ does not introduce such a branch.
\end{claim}
\begin{proof}
In order to prove the claim, we are going to find a forcing notion $\hat{\mathbb{Q}}$ and a projection from $\hat{\mathbb{Q}}$ to $\mathbb{Q}$. We will show that $\mathbb{T}(\kappa)\ast \hat{\mathbb{Q}} \ast j(\mathbb{S}(\kappa))$ (that projects to $\mathbb{T}(\kappa)\ast \mathbb{Q} \ast j(\mathbb{S}(\kappa))$) does not introduces new branches to $P_{\omega_2} \lambda$-trees, assuming that there is no branch that meets the ladder system cofinally. 

First, since $j(\mathbb{S}(\kappa))$ is forced to be $j(\kappa)$-strategically closed in the generic extension of $N$, it is forced to be at least $\lambda$-strategically closed in $V$. Thus, if $\dot{b}$ is a name for a new branch through $\mathcal T$ which is forced to meet $\mathcal L$ $\omega_1$-cofinally, then one can construct a filter of $j(\mathbb{S}(\kappa))$ deciding the value of $\dot{b}(\alpha)$ for all $\alpha < \lambda$, such that the obtained branch $b'$ indeed meets $\mathcal L$ cofinally.

Now, take $\hat{\mathbb{Q}}$ to be $\Add(\omega, j(\kappa) \setminus \kappa) \times \mathbb{C}$ where $\mathbb{C}$ is the collection of all conditions of the form $(1, m, s, t) \in \mathbb{Q}$, ordered by their induced order from $\mathbb{Q}$. Note that this is just the termspace forcing for $\mathbb{Q}$, and that $\mathbb{C}$ is $\sigma$-closed. See \cite[Section 22]{CummingsHandbook} for further details about termspaces and projections.

Since the forcing $\Add(\omega, j(\kappa) \setminus \kappa)$ is productively c.c.c., it cannot add branches to a $P_{\omega_2} \lambda$-tree, \cite[Lemma 1.6]{Unger2015}. Thus, any new branch was already introduced by $\mathbb{T}(\kappa)\ast \mathbb{C}$.

Let us assume that there is such a branch. Let $\mathcal{M}$ be a countable elementary substructure of $H(\chi)[G]$ that contains the forcing notions $\mathbb{T}, \mathbb{C}$, the tree and the name for the new branch $\dot{b}$. Let us pick $\mathcal M$ such that $\delta = \sup (\mathcal M \cap \kappa)$ does not belong to the set $S = \{\alpha < \kappa \mid \exists x \in \dom f,\, \sup x= \alpha\}$. There is such a model since the set $S$ is co-stationary on $S^{\omega_2}_\omega$.

Let us construct a prefect tree of mutually $\mathcal M$-generic filters, $\langle K_\eta \mid \eta \in {}^\omega 2\rangle$. Each one of those filters gives rise to a condition $\langle t_\eta, q_\eta\rangle$. For each $\eta$, $t_\eta = \bigcup \{t \mid \langle t, q\rangle \in K_\eta\} \cup \{\delta\} \in \mathbb{T}(\kappa)$ since $\delta \notin S$. For each $\eta$, the condition $q_\eta$ exists by the $\sigma$-closure of $\mathbb{C}$.

Now, for each $\eta \in {}^\omega 2$, there is a different realization of $\dot{b}$ on $\mathcal M$. Note that $\langle t_\eta, q_\eta\rangle$ forces the value of $\dot{b} \cap \mathcal M$ to be some $x_\eta$. By mutual genericity of the filters $K_\eta$, and since $\Vdash \dot{b}\notin V[G]$, for every $\eta \neq \eta'$, $x_{\eta} \neq x_{\eta'}$. But in this model $2^{\aleph_0} = \omega_2$, which contradicts the assumption that each level of the tree has size $<\omega_2$.
\end{proof}
Finally, let us show that the set 
\[B = \{x \in P_{\omega_1} \lambda \mid b \restriction x \in \bigcup \range \ell\}\]
is unbounded. Indeed, this set is even stationary as in $N[G'][S']$ (in which $\cf \kappa > \omega$) this set contains a club. 
\end{proof}
This establishes Theorem \ref{thm: catching unbounded but not club}.

As the different variants of the strong tree property behave differently on $\omega_2$, let us compare them to the Ineffable Tree Property. The model of Lemma \ref{lemma: catching unbounded but not clubs}, assuming full supercompactness, also provides the following separation result.
\begin{remark}
In the model of Theorem \ref{thm: catching unbounded but not club}, $ITP(\omega_2)$ holds. In particular, $ITP(\omega_2)$ is consistent the failure of $\LSCP(P_{\omega_2}\omega_2, \omega_1\text{-club})$.
\end{remark}
\begin{proof}
We work with full supercompact embeddings. Let $j \colon V \to M$ be a $\lambda$-supercompact embedding. As in the proof of Lemma \ref{lemma: catching unbounded but not clubs}, we can lift it to an elementary embedding $j^* \colon V[G] \to M[H]$.  


Let us consider now a $P_{\omega_2} \lambda$-tree $\mathcal{T}$ with a list $d$. Let us consider the branch $b$ which is generated by $j^*(d)(j \image \lambda) \in M[H]$. By the arguments of Lemma \ref{lemma: catching unbounded but not clubs}, this branch appears already in $V[G]$. We need to show that it is ineffable. Working in $V[G]$, let $B = \{x \in P_{\omega_2} \lambda \mid b(x) = d(x)\}$. If $B$ is non-stationary in $V[G]$, then there is a club $D$, avoiding it. Let us consider $j^*(D)$. $j\image \lambda = \bigcup_{x\in D} j^*(x) \in j^*(D)$. Therefore, $j \image \lambda \notin j^*(B)$, but this is absurd, as \[j^*(b)(j\image \lambda) = \bigcup_{x \in P_{\omega_2}\lambda} j^*(b(x)) = j^*(d)(j\image \lambda).\]
\end{proof}
\section{Questions}
We conclude the paper with some questions. 
Our model of Theorem \ref{thm: catching unbounded but not club} gives an unsatisfying separation between the different ladder system principles as the cofinal branch meets the ladder system on a stationary set, and not merely an unbounded set. This seems to be essential in this type of argument.
\begin{question}
Is it consistent that $\LSCP(P_{\omega_2}\lambda, \omega_1\emph{-cofinal})$ holds, but the seemly stronger property $\LSCP(P_{\omega_2}\omega_2, \omega_1\emph{-stationary})$ fails, namely there is an $\omega_2$-tree with a ladder system such that no branch branch meets that ladder system on a stationary set in $P_{\omega_1}\lambda$?
\end{question} 

\begin{question}
Does the $\LSCP(P_{\omega_2}\lambda, \omega_1\text{-clubs})$ imply the Ineffable Tree Property at $P_{\omega_2}\lambda'$ for some $\lambda' < \lambda$? 
\end{question}
\section{Acknowledgments}
We would like to thank the anonymous referee for reading the paper carefully and providing many useful suggestions and corrections as well as referring us to some highly relevant literature. 
\providecommand{\bysame}{\leavevmode\hbox to3em{\hrulefill}\thinspace}
\providecommand{\MR}{\relax\ifhmode\unskip\space\fi MR }
\providecommand{\MRhref}[2]{%
  \href{http://www.ams.org/mathscinet-getitem?mr=#1}{#2}
}
\providecommand{\href}[2]{#2}

\end{document}